\documentclass[12pt, leqno]{amsart}
\usepackage[active]{srcltx}
\usepackage{verbatim}
\usepackage{epsfig,graphicx,color,mathrsfs}
\usepackage{graphicx}
\usepackage{amsmath,amssymb,amsthm,amsfonts}
\usepackage{amssymb}
\usepackage[english]{babel}
\usepackage{appendix}

\usepackage[left=2.7cm,right=2.7cm,top=3cm,bottom=3cm]{geometry}

\newcommand{\R}{{\mathbb R}}

\newcommand{\rn}{{\mathbb{R}^N}}

\numberwithin{equation}{section}
\newtheorem{theorem}{Theorem}[section]
\newtheorem{proposition}[theorem]{Proposition}
\newtheorem{lemma}[theorem]{Lemma}

\newtheorem{definition}[theorem]{Definition}
\newtheorem{remark}[theorem]{Remark}
\theoremstyle{definition}

\newcommand{\brm}{\begin{remark}\rm}
\newcommand{\erm}{\end{remark}}
\newcommand{\brms}{\begin{remark}\rm}
\newcommand{\erms}{\end{remark}}
\newcommand{\bte}{\begin{theorem}}
\newcommand{\ete}{\end{theorem}}
\newcommand{\bpr}{\begin{proposition}}
\newcommand{\epr}{\end{proposition}}
\newcommand{\ble}{\begin{lemma}}
\newcommand{\ele}{\end{lemma}}
\newcommand{\beq}{\begin{equation}}
\newcommand{\eeq}{\end{equation}}
\newcommand{\bdm}{\begin{displaymath}}
\newcommand{\edm}{\end{displaymath}}
\numberwithin{equation}{section}

\newcommand{\bos}{\begin{remark}\rm}
\newcommand{\eos}{\end{remark}}

\newcommand{\ben}{\begin{enumerate}}
\newcommand{\een}{\end{enumerate}}

\newcommand{\e }{\varepsilon }

\newcommand{\be}{\begin{equation}}
\newcommand{\ee}{\end{equation}}

\newcommand{\inn}{\text{  in   }}

\title[Moving plane method]{On the moving plane method for nonlocal problems in bounded domains}

\author[B. Barrios]{ Bego\~na Barrios $^+$}

\author[L.\ Montoro]{Luigi Montoro$^*$}

\author[B.\ Sciunzi]{Berardino Sciunzi$^*$}



\thanks{$^+$ Departamento de Matem\'aticas, Universidad Aut\'onoma
de Madrid. 28009. Spain.
E-mail: {\em bego.barrios@uam.es}}

\thanks{$^*$Dipartimento di Matematica e Informatica,
Universit\`a della Calabria,
Ponte Pietro Bucci 31B, I-87036 Arcavacata di Rende, Cosenza, Italy,
E-mail:  {\em montoro@mat.unical.it}, {\em sciunzi@mat.unical.it}}

\thanks{BB and LM were partially supported by project MTM2010-18128, MICINN}

\thanks{BS  was partially supported by ERC-2011-grant: \emph{Elliptic PDE's and symmetry of interfaces and layers for odd nonlinearities.}}

\thanks{LM and BS  were partially supported by PRIN-2011: {\em Variational and Topological Methods in the Study of Nonlinear Phenomena}}

\begin{document}
\begin{abstract}
We consider a nonlocal problem  involving the fractional laplacian and the Hardy potential, in  bounded smooth domains. Exploiting the moving plane method and some weak and strong comparison principles,
we deduce symmetry and monotonicity properties of positive solutions
under zero Dirichlet boundary conditions.
\end{abstract}

\maketitle
\tableofcontents

\medskip

\section{Introduction}\label{introdue}

{In recent years, considerable attention has been given to equations involving general integrodifferential
operators, especially, those with the fractional Laplacian operator. This motivation coming from the fact that these nonlocal structures has connection with many real world phenomena. Indeed, non local operators naturally appear in elasticity problems \cite{signorini}, thin obstacle problem \cite{Caf79}, phase transition \cite{AB98, CSM05, SV08b}, flames propagation \cite{crs}, crystal dislocation \cite{fdpv, toland}, stratified materials \cite{savin_vald}, quasi-geostrophic flows \cite{caf_vasseur} and others. Since these operators are also related to L\'{e}vy processes and have a lot of applications to mathematical finance, they have been also studied from a probabilistic point of view (see for example \cite{kass3, bogdan_by, kass1, ito, enrico}).}  We refer the readers to, for instance,  \cite{barrios2, barrios3, barrios4, barrios5, BjCaffFigalli, CafFigalli, CScpam, CS2011, CS2012, guida, quim, joaquim, serv, serv3} where existence of solutions and/or regularity of solutions  are studied for some nonlocal problems.

{In this paper we focus our attention in the following problem}
\begin{equation}\label{Eq:P}
\left\{
\begin{array}{rcl}
(-\Delta)^s u&=&\frac{g(u)}{|x|^{2s}}+f(x,u)\quad\text{in}\,\, \Omega,\\
u&>&0\,\qquad\qquad\qquad\inn \Omega,\\
u&=&0\,\qquad\qquad\qquad\inn \rn\setminus\Omega,\\
\end{array}
\right.
\end{equation}
where $\Omega$ is a bounded smooth domain, $N>2s$ and the equation is understood in the weak energy sense (see Definition \ref{def:sol}) {and $(-\Delta)^s$ is the fractional Laplacian operator defined, up to a normalization factor by the Riesz potential as
$$(-\Delta)^s u(x):=
\int_{\mathbb{R}^N}\frac{2u(x)-u(x+y)-u(x-y)}{|y|^{N+2s}}\,dy,\quad x\in \mathbb{R}^N\,,
$$
where $0<s<1$ is a fix parameter (see \cite[Chapter 5]{Stein} or \cite{guida, PhS} for further details).} We assume that  $$0\in \Omega\,$$ and also that the nonlinearities
$$f(x,t):\Omega \times[0,\infty)\rightarrow \mathbb R\quad \text{and} \quad g(t): [0,\infty)\rightarrow\mathbb{R},$$ fulfill the following assumptions:
\begin{enumerate}
\item[$(H_1)$] $f(x,t)$ is a Carath\'{e}odory function which is locally Lipschitz continuous with respect to the second variable. Namely, for any $M>0$ given, it follows
\[
|f(x,t_1)-f(x,t_2)|\leq L_f(M)|t_1-t_2|,\quad x\in\Omega,\quad t_1,t_2\in[0\,,\,M].
\]
Furthermore
$g(t)$ is locally Lipschitz continuous namely, for any $M>0$ given, it follows
\[
|g(t_1)-g(t_2)|\leq L_g(M)|t_1-t_2|,\quad  t_1,t_2\in[0\,,\,M].
\]
\end{enumerate}
As a leading example we can consider $f(x,s)=a(x)f(s)$ with $a(\cdot)$ bounded and measurable and $f$ locally Lipschitz continuous and  $g(t)=\vartheta t^q$ with $\vartheta\geqslant0$ an $q\geqslant 1$.{We note here that, adapting to the nonlocal framework the ideas done in \cite[Theorem 1.2]{brezcabre}, if $f\geq 0$ and $g(u)\geq u^{q}, \, q>1$ it can be proved that the problem has not solution even in a more weaker sense than the one considered in Definition \ref{def:sol}.

\

In some  of our results we will consider the stronger assumption:

\

 \begin{enumerate}
\item[$(H_2)$] $f(x,t)$ is Holder continuous with respect to the $x$-variable, namely, for any $M>0$ and ${r}>0$ given, for some $\gamma\in(0,1)$ it follows that
\[
|f(x_1,t)-f(x_2,t)|\leq C(M,r)|x_1-x_2|^\gamma,\quad x_1,x_2\in\Omega\setminus B_{{r}}(0),\quad t\in[0\,,\,M].
\]
\end{enumerate}
 The aim of this work is to  prove symmetry and monotonicity properties of the solutions exploiting the \emph{moving plane method}. The moving plane method was brought to the attention of the PDE community by J. Serrin {(\cite{S})} and  a clever use of it goes back to the celebrated paper \cite{GNN}. The technique was refined in \cite{BN} and this is the approach that we use here. This  will also allows us to consider convex (not necessarily strictly convex domains).\\
\noindent The general statement is the following:\\

\noindent \emph{If the domain is convex and symmetric, then the solution inherits the symmetry of the domain and also
exhibits monotonicity properties.
}\\

 When performing the  moving planes technique in problems that involves local partial differential equations, the local properties of the differential operators are used in a crucial way. This causes that,  in the context of nonlocal operators, many difficulties arise, for example, because of the lack of general weak and strong comparison principles. Previous contribution devoted to symmetry results for equations involving the fractional Laplacian in $\mathbb{R}^{N}$ that use the moving plane method can be found, for instance, in \cite{CLO, CLO2, DMPS, FeWa, MC}. Other woks, in the nonlocal framework, that study the symmetry of solutions using another techniques are for example \cite{DPV, CSire, SV08b}.
\\

The analysis in our context is also more involved because of the presence of the Hardy Leray potential. In particular this causes that the solutions are not bounded (and not smooth) near the origin. Nevertheless, the case  $g(u)=0$
is also admissible in our results and in this case our effort is to carry out the moving plane procedure exploiting the weak formulation of the equation. This allows to consider issues where solutions are not smooth, namely not of class $\mathcal{C}^1$.

Our main result is the following:

\begin{theorem}\label{symmetryintro}
Let $u\in \mathcal{C}^0(\overline{\Omega}\setminus\{0\})$ be a weak solution to \eqref{Eq:P} and let $\Omega$ be  convex with respect to the $x_1$-direction  and symmetric w.r.t.
$$ T_0=\{x\in \mathbb{R}^N : x_1=0\}.$$
For $x=(x_1,x')$ let {us consider} $x_\lambda=(2\lambda-x_1,x')$ and assume that
{$$\mbox{$f(x,t)\leq f(x_\lambda,t)$ if $\lambda <0$, $x\in\Omega\cap\{x_1<\lambda\}$ and $t\in [0,\infty)$}$$
and
$$\mbox{$f(x,t)= f(x_\lambda,t)$ if $x\in\Omega\cap\{x_1<0\}$ and $t\in [0,\infty)$.}$$}

\noindent Let us also assume that either $(a)$ or $(b)$ are fulfilled, where:

\begin{itemize}
\item [$(a)$]  $(H_1)$  holds  and  $f(x,t)$  is  nondecreasing with respect to the variable $t$ for any $x\in \Omega$ and $g(t)$  is  nondecreasing with respect to the variable $t$;

\item[$(b)$]  $(H_1)$ and  $(H_2)$ hold.
\end{itemize}

\

\noindent Then $u$ is symmetric w.r.t. the $x_1$-variable and strictly increasing w.r.t. the $x_1$-direction for $x_1<0$.
Moreover, if  $\Omega$ is a ball, then $u$ is radial and strictly radially decreasing.
\end{theorem}

The fact that we need to assume some monotonicity and symmetry properties of the nonlinearity is natural. In fact it is easy to see that, if the right hand side in our problem is not symmetric, then  the solution cannot inherit the symmetry of the domain since we know that, for instance, $(-\Delta)^{s}|x|^{-\beta}= C(N,s,\beta)|x|^{-\beta-2s},\, \beta>0$. Furthermore, also the monotonicity in variable $x$ is necessary. In the local framework, this can be deduced considering e.g. the Hen\"{o}n equation for which non-radial solutions do exist. \\

A particular but relevant  example for which Theorem \ref{symmetryintro} applies is the following 
\begin{equation}\label{Eq:P222}
\left\{
\begin{array}{rcl}
(-\Delta)^s u&=&\frac{u}{|x|^{2s}}+f(u)\quad\text{in}\,\, \Omega,\\
u&>&0\,\qquad\qquad\qquad\inn \Omega,\\
u&=&0\,\qquad\qquad\qquad\inn \rn\setminus\Omega,\\
\end{array}
\right.
\end{equation}
with $f$ locally Lipschitz continuous with critical or sub-critical growth. It is easy to check that all our assumptions are fulfilled in this case. Furthermore, the case $g(\cdot)=0$ is also admissible in our result and we have in this case the following:

\begin{theorem}\label{symmetryintrokfhgkdhk}
Let $u\in \mathcal{C}^0(\overline{\Omega})$ be a weak solution to
 \begin{equation}\label{Eq:P333}
\left\{
\begin{array}{rcl}
(-\Delta)^s u&=&f(x,u)\qquad\qquad\text{in}\,\, \Omega,\\
u&>&0\,\qquad\qquad\qquad\inn \Omega,\\
u&=&0\,\qquad\qquad\qquad\inn \rn\setminus\Omega\\
\end{array}
\right.
\end{equation}
 and let $\Omega$ be  convex with respect to the $x_1$-direction  and symmetric w.r.t.
$$ T_0=\{x\in \mathbb{R}^N : x_1=0\}.$$
For $x=(x_1,x')$ let {us consider} $x_\lambda=(2\lambda-x_1,x')$ and assume that
{$$\mbox{$f(x,t)\leq f(x_\lambda,t)$ if $\lambda <0$, $x\in\Omega\cap\{x_1<\lambda\}$ and $t\in [0,\infty)$}$$
and
$$\mbox{$f(x,t)= f(x_\lambda,t)$ if $x\in\Omega\cap\{x_1<0\}$ and $t\in [0,\infty)$.}$$}

\noindent Let us also assume that either $(a)$ or $(b)$ are fulfilled, where:

\begin{itemize}
\item [$(a)$]  $(H_1)$  holds (with $g=0$)  and  $f(x,t)$  is  nondecreasing with respect to the variable $t$ for any $x\in \Omega$;

\item[$(b)$]  $(H_1)$ and  $(H_2)$ hold (with $g=0$).
\end{itemize}

\

\noindent Then $u$ is symmetric w.r.t. the $x_1$-variable and strictly increasing w.r.t. the $x_1$-direction for $x_1<0$.
Moreover, if  $\Omega$ is a ball, then $u$ is radial and strictly radially decreasing.
\end{theorem}
Theorem \ref{symmetryintrokfhgkdhk} extends to the nonlocal context the classical results of \cite{BN,GNN}, that hold in the local case. Even in this case, namely when the Hardy potential is not considered, our results and their proofs are new. In fact, we perform the technique exploiting only the weak formulation of the equation. This allows to consider the case when the solution is merely continuous.\\

\noindent The proof of Theorem \ref{symmetryintro}  will be carried out via the moving plane method. To do this we need
to exploit weak and strong comparison principles. The weak comparison principle cannot hold in general and in fact we will prove and exploit a \emph{weak comparison principle in small domains} in Theorem \ref{wcpsmall}. {This is because, under the assumption $(a)$ a lack of regularity of the solutions, force us to use} a careful analysis. {Moreover} we need to have a precise control of the parameters involved in the \emph{weak comparison principle in small domains} so that the latter could be of use when performing the moving plane procedure. In particular we have to take into account the fact that solutions are unbounded at the origin.\\

\noindent The other important tool is the \emph{strong comparison principle}. When $f(x,t)$ and $g(t)$ are nondecreasing w.r.t. the $t$-variable, we succeed in exploiting earlier results in \cite{PhS} as slightly improved in \cite{DMPS}. If this is not the case, namely considering $(b)$ in Theorem \ref{symmetry}, we argue in a different way and write the equation pointwise  far from the origin. To do this we need some  regularity information more, that will be deduced by the assumption $(H_2)$ {(see Proposition \ref{pr:regularity})}.\\

\noindent The paper is organized as follows: in Section 2, an introduction of the necessary functional framework is presented, as well as the type of solution we will work with and an interior regularity result. Section 3 is devoted to prove the weak and strong comparison principle. These results are the fundamental key to apply, in Section 4, the moving plane method to obtain the symmetry of the solutions.

\section{Notations and Preliminary Results}

\noindent Let us recall that, given  a function $u$ in the Schwartz's class $\mathcal{S}(\mathbb{R}^{N})$ we define for $0<s<1$, the fractional Laplacian as
\begin{equation}\label{fourier1}
\widehat{(-\Delta)^{s}}u(\xi)=|\xi|^{2s}\widehat{u}(\xi),\quad \xi\in\mathbb{R}^{N}.
\end{equation}
It is well known (see \cite{Landkof, Stein, enrico}) that this operator can be also represented, for suitable functions, as a principal value of the form
\begin{equation}\label{org1}
(-\Delta)^s u(x):=c_{N,s}\,{\rm P.V.}\int_{\mathbb{R}^N}\frac{u(x)-u(y)}{|x-y|^{N+2s}}\,dy
\end{equation}
where
\begin{equation}\label{constante}
c_{N,s}:=\left(\int_{\mathbb{R}^{N}}{\frac{1-\cos(\xi_1)}{|\xi|^{N+2s}}\, d\xi}\right)^{-1}=\frac{4^s\Gamma\left(\frac{N}{2}+s\right)}{-\pi^{\frac{N}{2}}\Gamma(-s)}>0,
\end{equation}
is a normalizing constant chosen to guarantee that \eqref{fourier1} is satisfied (see \cite{guida, PhS, enrico}).
From \eqref{org1} one can check that
\begin{equation}\label{decaimiento}
|(-\Delta)^{s}\phi(x)|\leq \frac{C}{1+|x|^{N+2s}},\quad\mbox{for every $\phi\in\mathcal{S}(\mathbb{R}^{N})$.}
\end{equation}
This motivates the introduction of the space
$$
\mathcal{L}^{s}(\mathbb{R}^{N}):=\{u:\mathbb{R}^{N}\to\mathbb{R}:\quad \int_{\mathbb{R}^{N}}{\frac{|u(x)|}{(1+|x|^{N+2s})}\, dx}<\infty\},
$$
endowed with the natural norm
$$\|u\|_{\mathcal{L}^{s}(\mathbb{R}^{N})}:=\int_{\mathbb{R}^{N}}{\frac{|u(x)|}{(1+|x|^{N+2s})}\, dx}.$$
Then, if $u\in\mathcal{L}^{s}(\mathbb{R}^{N})$ and $\phi\in\mathcal{S}(\mathbb{R}^{N})$, using \eqref{decaimiento}, we can formally define the duality product $\langle (-\Delta)^{s}u, \phi\rangle$ in the distributional sense as
$$\langle (-\Delta)^{s}u, \phi\rangle:=\int_{\mathbb{R}^{N}}{u(-\Delta)^s \phi\, dx}.$$

Along this work we will consider the Hilbert space
$$H^s_0(\Omega)=\{u\in H^s(\mathbb{R}^N)\,:\, u=0\,\,\text{in}\,\, \mathbb{R}^N\setminus\Omega\},$$
endowed with the norm
$$\|u\|^2_{H_0^s(\Omega)}=\frac{2}{c_{N,s}}\|(-\Delta)^{s/2}u\|^2_{L^2(\mathbb{R}^N)}.$$
Here $c_{N,s}$ is the normalizing constant given in \eqref{constante}.

In the following we will exploit the following well known Sobolev-type embedding Theorem
\begin{theorem}{\rm(See \cite[Theorem 7.58]{adams}, \cite[Theorem~6.5]{guida}, \cite{lieb, stein})}\label{embeding}
Let $0<s<1$ and $N>2s$. There exists a constant $S_{N,s}$ such that, for any measurable and compactly supported function
$u:\mathbb{R}^{N}\to\mathbb{R}$, we have
$$S_{N,s}\|u\|_{L^{2^{*}_s}(\mathbb{R}^{N})}^{2}\leq
\frac{2}{c_{N,s}}\|(-\Delta)^{s/2}u\|^2_{L^2(\mathbb{R}^N)},$$
where
\begin{equation}\label{critical_sobolev}
2^*_s=\frac{2N}{N-2s},
\end{equation}
is the Sobolev critical exponent.
\end{theorem}

Now we are in position to give the following:

\begin{definition}\label{def:sol}
We say that $u\in H^s_0(\Omega)$  is a weak solution to \eqref{Eq:P} if:
\[
\frac{g(u)}{|x|^{2s}}\in L^1(\Omega)\qquad\text{and}\qquad f(x,u)\in L^1(\Omega)
\]
and

\begin{equation}\nonumber
\begin{split}
\frac 12 c_{N,s}\int_{Q}&\frac{(u(x)-u(y))(\varphi(x)-\varphi(y))}{|x-y|^{N+2s}}\,dx\,dy\\
&=\int_{\Omega}\frac{g(u)}{|x|^{2s}}\varphi \,dx+\int_{\Omega}f(x,u)\varphi\,dx\quad\forall \varphi \in H^s_0(\Omega)\cap L^\infty(\Omega)\,.
\end{split}
\end{equation}
where $c_{N,s}$ has been defined in \eqref{constante} and $Q=\mathbb{R}^{2N} \setminus(\Omega^{c} \times \Omega^{c})$.
\end{definition}

\begin{remark}
We point out that, using  the fractional Hardy-Sobolev inequality (see \cite{beckner, fls, herb}),
it follows that
$\frac{g(u)}{|x|^{2s}}\in L^1(\Omega)$ whenever $g(\cdot)$ has linear or sublinear  growth at infinity, so that Definition \ref{def:sol} makes sense in this case. Actually, in this case, it follows that $\frac{g(u)\varphi}{|x|^{2s}}\in L^1(\Omega)$ for any (possibly unbounded) $\varphi\in H^s_0(\Omega)$. Therefore it is possible in this case to consider also unbounded test functions in the weak formulation of the equation.\\
\noindent Furthermore, by the Sobolev embedding, the case when $f$ has critical or sub-critical growth is also admissible even without the bounded condition in the family of test functions.
\end{remark}

Relating to some properties of the fractional Laplacian operator we present here a regularity result that will be needed later.
\begin{proposition}(Regularity-Bootstrap)\label{pr:regularity}
Let $u\in H^s_0(\Omega)$ be a weak solution to \eqref{Eq:P}. Let us consider $B_{r}(x_0)\subseteq\Omega$ for some $r>0$ and $x_0\in\Omega$ such that $0\notin B_{r}(x_0)$. If we denote by $h(x,u)$ the right hand side of the equation given in \eqref{Eq:P} and we assume that $u\in L^{\infty}(B_r(x_0))$, then
\begin{itemize}
\item[i)]$u\in\mathcal{C}^{\beta}\left(\overline{B_{\frac{r}{4}}(x_0)}\right)$ for every $0<\beta<2s$ if $f(x,u)$ and $g(u)$ satisfies the condition $(H_1)$.
\item[ii)]Moreover,  if the condition $(H_2)$ is also verified, $u\in \mathcal{C}^{\alpha+2s}\left(\overline{B_{\frac{r}{16}}(x_0)}\right)$ for some $0<\alpha<2s$, $\alpha\notin\mathbb{Z}$ and $\alpha+2s\notin\mathbb{Z},$ and, in fact,
\begin{equation}\label{regularidad}
\|u\|_{\mathcal{C}^{\alpha+2s}\left(\overline{B_{\frac{r}{16}}(x_0)}\right)}\leq c\left(\|u\|_{\mathcal{L}^{s}(\mathbb{R}^{N})}+\|u\|_{L^{\infty}(B_{r}(x_0))}+\|h\|_{L^{\infty}(B_{r}(x_0))}+\|h\|_{\mathcal{C}^{\alpha}\left(\overline{B_{\frac{r}{4}}(x_0)}\right)}\right),
\end{equation}
where $c$ is a positive constant that only depends on $N,s$ and $\alpha$.
\end{itemize}

\end{proposition}
\begin{proof}
First of all we observe that, since $u\in L^{\infty}(B_r(x_0))$,  by hypothesis $(H_1)$, we obtain that
\begin{equation}\label{regularidad1}
h(x,u):=\frac{{g(u)}}{|x|^{2s}}+f(x,u){\in L^{\infty}(B_r(x_0)).}
\end{equation}
Then, by Theorem \ref{embeding} and by the fact that $\Omega$ is a bounded domain, we get that
\begin{eqnarray}
\|u\|_{\mathcal{L}^{s}(\mathbb{R}^{N})}&\leq& C\left(\|u\|_{L^{1}(B_1)}+\int_{\mathbb{R}^{N}\setminus B_1}\frac{|u(x)|}{|x|^{N+2s}}\, dx\right)\nonumber\\
&\leq&C\left(|\Omega|^{\frac{1}{(2^*_s)'}}\|u\|^{2^*_s}_{L^{2^*_s}(\Omega)}+\int_{\Omega\setminus B_1}\frac{|u(x)|}{|x|^{N+2s}}\, dx\right)\nonumber\\
&<&\infty.\label{regularidad2}
\end{eqnarray}
Therefore by \cite[Corollary 2,5]{quim} it follows that, for every $0<\beta<2s$,
\begin{equation}\label{regularidad3}
\|u\|_{\mathcal{C}^{\beta}\left(\overline{B_{\frac{r}{4}}(x_0)}\right)}\leq c\left(\|u\|_{\mathcal{L}^{s}(\mathbb{R}^{N})}+\|u\|_{L^{\infty}(B_{r}(x_0))}+\|h\|_{L^{\infty}(B_{r}(x_0))}\right),
\end{equation}
where $c$ is a positive constant that only depends on $N,s$ and $\beta$. Since, by $(H_1)$ and $(H_2)$, we know that {$h\in\mathcal{C}^{\alpha}\left(\overline{B_{\frac{r}{4}}(x_0)}\right)$}, with $\alpha=min\{\beta, \gamma\}$, where $\gamma$ was given in $(H_2)$, by \eqref{regularidad2} and  \eqref{regularidad3}, we obtain that \eqref{regularidad} follows by applying \cite[Corollary 2,4]{quim},
\end{proof}
\begin{remark}\label{rm:point}
Observe that as a consequence of the previous result, when $u$ is a weak solution, $(H_2)$ is assumed,  for any {$x\in\Omega\setminus \{0\}$}, if $u$ is bounded in a neighborhood of $x$, we can write in a pointwise way
$$c_{N,s} \,\,P.V. \int_{\mathbb{R}^N} \frac{u(x)-u(y)}{|x-y|^{N+2s}}dy= \frac{g(u)}{|x|^{2s}}+f(x,u).$$
\end{remark}

To finish this section we introduce some notation that we will need to state the principal results of the work.
If $\nu$ is a direction in $\mathbb{R}^N$, i.e. $\nu  \in \mathbb{R}^N$ and  $|\nu|=1$, and $\lambda$ is  a real number  we set
$$T_\lambda^\nu:=\{x\in \mathbb{R}^N:x\cdot\nu=\lambda\}.$$
Moreover, let us denote
$$
\Sigma_\lambda^\nu:=\{x\in \mathbb{R}^N:x\cdot \nu <\lambda\},\qquad \Omega_\lambda^\nu:=\Omega\cap\Sigma_\lambda^\nu
$$
$$
x_\lambda^\nu=R_\lambda^\nu(x):=x+2(\lambda -x\cdot\nu)\nu,
$$
(i.e. $R_\lambda^\nu $  is the reflection trough the hyperplane $T_\lambda^\nu$),
\begin{equation}\label{eq:sn33}
u_{\lambda }^\nu(x):=u(x_\lambda^\nu) \,
\end{equation}
and
$$
a(\nu):=\inf _{x\in\Omega}x\cdot \nu.
$$
When $\lambda >a(\nu)$, since $\Omega_\lambda^\nu$ is nonempty, we set
$$
(\Omega_\lambda^\nu)':=R_\lambda^\nu(\Omega_\lambda^\nu)
$$
and, finally, for $\lambda>a(\nu)$ we denote
\begin{equation}\label{eq:sn7}
\lambda_1(\nu):=\sup\{\lambda : (\Omega_\lambda^\nu)'\subset \Omega\}.\end{equation}

\noindent {\bf Notation.} Generic fixed and numerical constants will be denoted by
$C$ (with subscript in some case) and they will be allowed to vary within a single line or formula. By $|A|$ we will denote the Lebesgue measure of a measurable set $A$.

\section{Comparison principles}
Now we prove a weak comparison theorem in small domain, namely we have the following
\begin{theorem}[Weak comparison principle in small domains] \label{wcpsmall}
Let $\lambda <0$ and let us consider a set ${\widetilde D}$ such that ${\widetilde D}\subseteq  \Omega_\lambda^\nu \subset \Sigma_\lambda^\nu$.
Moreover let $u,v \in H^s(\mathbb{R}^N)$   weakly satisfying
\begin{equation}\label{ec_uu}
(-\Delta)^s u\leq\frac{g(u)}{|x|^{2s}}+f(x,u),\quad\text{in}\,\, {\widetilde D}
\end{equation}
\begin{equation}\label{ec_v}
(-\Delta)^s v\geq\frac{g(v)}{|x|^{2s}}+f(x,v),\quad\text{in}\,\, {\widetilde D},
\end{equation}
with $u\in L^\infty(\Omega_\lambda^\nu)$,  $f$ and $g$ satisfying $(H_1)$. Assume that $u\leq v$ in $\Sigma_{\lambda}^{\nu}\setminus \widetilde D$ and   $\big(u- v\big )$ is odd with respect to $T_\lambda^\nu=\partial\Sigma_\lambda^\nu$. Then there exists
\begin{equation}\label{controlling}
\delta=\delta\left(s,N,\lambda,L_f(\|u\|_{L^\infty(\Omega_\lambda^\nu)}),L_g(\|u\|_{L^\infty(\Omega_\lambda^\nu)})\right),
\end{equation}
 such that if  we assume that
$|{\widetilde D}|\leq \delta$,  then
$$ u\leq v\,\, \text{in}\,\, {\widetilde D}$$
and actually in $\Sigma^{\nu}_\lambda$.
\end{theorem}
\begin{proof}
Taking into account $(H_1)$, we define
\[
\hat L_f(\lambda)\,:=\, L_f(\|u\|_{L^\infty(\Omega_\lambda^\nu)}) \qquad \mbox{and} \qquad\hat L_g(\lambda)\,:=\,L_g(\|u\|_{L^\infty(\Omega_\lambda^\nu)}).
\]
Let us set
\begin{equation}\label{defw}
w(x):=\left\{\begin{array}{ll}
(u-v)^+(x) & \quad {\mbox{ if }} x\in\Sigma^{\nu}_\lambda, \\
0 & \quad {\mbox{ if }}x\in\R^N\setminus\Sigma^{\nu}_\lambda,
\end{array} \right.
\end{equation}
where~$(u-v)^+:=\max\{u-v,0\}$. It follows that $w\in H_0^s({\widetilde D})$ (see e.g. \cite[Proposition 2.4]{LPPS}). Furthermore $w$ is bounded with $\|w\|_{L^\infty(\R^N)}\leqslant \|u\|_{L^\infty(\Omega_\lambda^\nu)}$ and then we can consider $w$ as a test function in \eqref{ec_uu} and \eqref{ec_v} obtaining that
$$
\frac{c_{N,s}}{2}\int_{\mathbb{R}^N}\int_{\mathbb{R}^N}\frac{(u(x)-u(y))(w(x)-
w(y))}{|x-y|^{N+2s}}\,dx\,dy\leq\int_{D}\frac{g(u)}{|x|^{2s}}w dx+\int_{D}f(x,u)w dx
$$
and
$$
\frac{c_{N,s}}{2}\int_{\mathbb{R}^N}\int_{\mathbb{R}^N}\frac{(v(x)-v(y))(w(x)-
w(y))}{|x-y|^{N+2s}}\,dx\,dy\geq\int_{D}\frac{g(v)}{|x|^{2s}}w dx+\int_{D}f(x,v)w dx.
$$
{where $D:= supp\, w\subseteq \widetilde{D}$}. Subtracting the two previous inequalities we get that
\begin{eqnarray}\nonumber
&&\frac 12 c_{N,s}\int_{\mathbb{R}^N}\int_{\mathbb{R}^N}\frac{\left((u(x)-v(x))-(u(y)-v(y))\right)
\left(w(x)-w(y)\right)}{|x-y|^{N+2s}}\,dx\,dy\\
&\leq&\int_{D}\left(\frac{g(u)-g(v)}{|x|^{2s}}\right)w dx+\int_{D}(f(x,u)-f(x,v))w dx. \label{eq:Mm2}
\end{eqnarray}
\noindent In the following it will be crucial the following remark:
\begin{equation}\label{eq:boundedsol}
\|v\|_{L^{\infty}({\widetilde D})}\leqslant \|u\|_{L^{\infty}(\Omega_\lambda^\nu)}\leqslant \bar C(\lambda)\,,
\end{equation}
for some positive constant $\bar C$.

Taking into account that $\lambda <0$, there exists a constant $C(\lambda)$ such that $|x|\geq C$ in $\Omega_\lambda^\nu$. Then from \eqref{eq:Mm2} we get
\begin{eqnarray}\nonumber
&&\frac 12 c_{N,s}\int_{\mathbb{R}^N}\int_{\mathbb{R}^N}\frac{\left((u(x)-v(x))-(u(y)-v(y))\right)
\left(w(x)-w(y)\right)}{|x-y|^{N+2s}}\,dx\,dy\\\nonumber
&\leq&C(\lambda)\int_{D}\left({g(u)-g(v)}\right)w dx+\int_{D}(f(x,u)-f(x,v))w dx\\ \label{eq:M2}
&\leq&\hat C(\lambda,\hat L_f(\lambda),\hat L_g (\lambda))\int_{D}w^2dx,
\end{eqnarray}
where we  have used \eqref{eq:boundedsol},  $(H_1)$   and the fact that $(u-v)w=w^2.$

On the other hand we have
\begin{equation}\label{eq:M3}
\begin{split}
&\int_{\mathbb{R}^N}\int_{\mathbb{R}^N}
\frac{\left((u(x)-v(x))-(u(y)-v(y))\right)\left(w(x)-w(y)\right)}{|x-y|^{N+2s}}\,dx\,dy\\
&=\int_{\mathbb{R}^N}\int_{\mathbb{R}^N}
\frac{\left(w(x)-w(y)\right)^2}{|x-y|^{N+2s}}\,dx\,dy\\
&+\int_{\mathbb{R}^N}\int_{\mathbb{R}^N}
\frac{\left(\left(u(x)-v(x))-(u(y)-v(y)\right)-\left(w(x)-w(y)\right)\right)\left(w(x)-w(y)\right)}{|x-y|^{N+2s}}\,dx\,dy\\
&=\int_{\mathbb{R}^N}\int_{\mathbb{R}^N}
\frac{\left(w(x)-w(y)\right)^2}{|x-y|^{N+2s}}\,dx\,dy+\int_{\mathbb{R}^N}\int_{\mathbb{R}^N}
\frac{\mathcal{A}(x,y)}{|x-y|^{N+2s}}\,dx\,dy,
\end{split}
\end{equation}
where
\[
\mathcal{A}(x,y)\,:=\,\left(\left(u(x)-v(x))-(u(y)-v(y)\right)-
\left(w(x)-w(y)\right)\right)\left(w(x)-w(y)\right).
\]

Now, we prove that
\begin{equation}\label{dfjhgjgfjfdecom}
\int_{\mathbb{R}^N}\int_{\mathbb{R}^N}
\frac{\mathcal{A}(x,y)}{|x-y|^{N+2s}}\,dx\,dy\geq 0\,.
\end{equation}
For that we will descompose the space as follows
\begin{equation}\nonumber
\begin{split}
\mathbb{R}^N\times \mathbb{R}^N\,=\,\left( D\cup \mathcal CD\cup R_\lambda^\nu(D)\cup
 \mathcal{C}R_\lambda^\nu(D)\right)\times \left( D\cup \mathcal CD\cup R_\lambda^\nu(D)\cup
 \mathcal{C}R_\lambda^\nu(D)\right),
\end{split}
\end{equation}
where
$$\mathcal CD:=\Sigma_{\lambda}^{\nu}\setminus D\quad\mbox{and}\quad \mathcal{C}R_\lambda^\nu(D):=\left(\mathbb{R}^{N}\setminus\Sigma_{\lambda}^{\nu}\right)\setminus R_\lambda^\nu(D).$$
Since  $u-v$ is odd with respect to $T_\lambda^{\nu}$ by assumption, we get that
\begin{center}
\begin{equation}\nonumber
\begin{split}
&\mathcal{A}(x,y)= \left[-\left(u(x)-v(x)\right)w(y)\right]\geq0\quad\text{in}\quad \left(\mathcal CD \times D\right)\\
&\mathcal{A}(x,y)= \left[-\left(u(y)-v(y)\right)w(x)\right]\geq0\quad\text{in}\quad \left(D \times \mathcal CD\right)\\
&\mathcal{A}(x,y)= \left[-\left(u(y)-v(y)\right)w(x)\right] \leq0\quad\text{in}\quad \left(D \times  \mathcal{C}R_\lambda^\nu(D)\right)\\
&\mathcal{A}(x,y)= \left[-\left(u(y)-v(y)\right)w(x)\right]\geq0\quad\text{in}\quad \left(D \times R_\lambda^\nu(D)\right)\\
&\mathcal{A}(x,y)= \left[-\left(u(x)-v(x)\right)w(y)\right]\geq0\quad\text{in}\quad \left(R_\lambda^\nu(D) \times D\right)\\
&\mathcal{A}(x,y)= \left[-\left(u(x)-v(x)\right)w(y)\right]\leq0\quad\text{in}\quad \left( \mathcal{C}R_\lambda^\nu(D) \times D\right)\\
&\mathcal{A}(x,y)= 0\qquad \qquad\qquad\qquad\qquad\quad\,\,\,\,\text{elsewhere}\,.
\end{split}
\end{equation}
\end{center}
Therefore,we immediately get that
\begin{equation}\label{nuevo}
\int_{R_\lambda^\nu(D)}\int_{D}
\frac{\mathcal{A}(x,y)}{|x-y|^{N+2s}}\,dx\,dy+
\int_{D}\int_{ R_\lambda^\nu(D)}
\frac{\mathcal{A}(x,y)}{|x-y|^{N+2s}}\,dx\,dy
\geq 0.
\end{equation}
Moreover for $(x,y)\in D\times \mathcal CD$, using again the fact that $u-v$ is odd with respect to $T_\lambda^{\nu}$ it follows that
\begin{eqnarray}
0\leq\mathcal A(x,y)&=&-\left(u(y)-v(y)\right)w(x)\nonumber\\
&=& \left(u(y_\lambda)-v(y_\lambda)\right)w(x)\nonumber\\
&=&-\mathcal A(x,y_\lambda).\label{equ1}
\end{eqnarray}
Therefore, since $|x-y|\leq|x-y_\lambda|$ and $\mathcal A(x,y)\geq0$ when $(x,y)\in D\times  \mathcal CD$, by \eqref{equ1} we get that
\begin{eqnarray}\nonumber
&&\int_{D}\int_{CD}
\frac{\mathcal{A}(x,y)}{|x-y|^{N+2s}}\,dx\,dy+
\int_{ D}\int_{\mathcal CR_\lambda^\nu(D)}
\frac{\mathcal{A}(x,y)}{|x-y|^{N+2s}}\,dx\,dy\\\nonumber
&&\qquad =\int_{D}\int_{CD}
\left(\frac{\mathcal{A}(x,y)}{|x-y|^{N+2s}}\,dx\,dy+
\frac{\mathcal{A}(x,y_\lambda)}{|x-y_\lambda|^{N+2s}}\right)\,dx\,dy\\
&&\qquad =\int_{ D}\int_{CD}
\mathcal A(x,y)\left(\frac{1}{|x-y|^{N+2s}}-\frac{1}{|x-y_\lambda|^{N+2s}}\right)\,dx\,dy\ge 0.\label{esti1}
\end{eqnarray}
Similarly, one can prove that
\begin{equation}\label{esti2}
\int_{\mathcal CD}\int_{D}
\frac{\mathcal{A}(x,y)}{|x-y|^{N+2s}}\,dx\,dy+
\int_{\mathcal CR_\lambda^\nu(D)}\int_{ D}
\frac{\mathcal{A}(x,y)}{|x-y|^{N+2s}}\,dx\,dy
\geq 0.
\end{equation}
Then, by \eqref{nuevo}, \eqref{esti1} and \eqref{esti2},  \eqref{dfjhgjgfjfdecom} follows. Hence by \eqref{eq:M2}, \eqref{eq:M3} and \eqref{dfjhgjgfjfdecom}, we get that
%
$$\frac{c_{N,s}}{2}\int_{\mathbb{R}^N}\int_{\mathbb{R}^N}\frac{(w(x)-w(y))^2}{|x-y|^{N+2s}}\,dx\,dy\le
\hat C(\lambda,\hat L_f(\lambda),\hat L_g(\lambda))\int_{D}w^2dx.$$
Moreover, using  H\"older inequality and Theorem \ref{embeding}, we get that
\begin{eqnarray}\nonumber
\frac{c_{N,s}}{2}\|w\|_{H_0^s(\widetilde{D})}^{2}
&\leq&  \hat C(\lambda,\vartheta,\hat L_f(\lambda),\hat L_g(\lambda))|D|^{\frac{2^*_s-2}{2^*_s}}\left (\int_{D}w^{2^*_s} dx\right)^{\frac{2}{2^*_s}}
\\
&\leq&\frac{\hat C(\lambda,\vartheta,\hat L_f(\lambda),\hat L_g(\lambda))}{S_{N,s}}|\widetilde D|^{\frac{2^*_s-2}{2^*_s}} \|w\|_{H_0^s(\widetilde{D})}^{2}.\label{ssant}
\end{eqnarray}
Choosing $\delta=\delta(s,N,\lambda,\hat L_f(\lambda),\hat L_g(\lambda))$ such that
$$\delta<\left(\frac{c_{N,s}S_{N,s}}{2\hat C(\lambda,\hat L_f(\lambda),\hat L_g(\lambda))}\right)^{\frac{2^*_s}{2^*_s-2}},$$
if $|\widetilde D|\leq \delta$, we obtain from \eqref{ssant}
$$\|w\|_{H_0^s(\widetilde{D})}^{2}=0.$$
Then we obtain that $w=0$ in $\mathbb{R}^{N}$, so, in particular $u\leq v$ in D. This clearly implies that $u\leq v$ in $\widetilde{D}$ and, moreover, in $\Sigma_{\lambda}^{\nu}$.
\end{proof}
{We state now the following strong comparison principle as follows}
\begin{proposition}[Strong comparison principle]\label{proconfronto}
Let $w\in H^s_0(\Omega)$ be a continuous solution of
\begin{equation}\label{sjdfbhswkvcgvcdgcvgdvc}
(-\Delta)^sw \geq 0\quad \text{in} \,\, D
\end{equation}
with $D\subset \Omega_\lambda^\nu$. If $w\geq 0$ in $\Sigma_\lambda^\nu$ and odd with respect to the hyperplane ${T_{\lambda}^{\nu}}$ then
$$w>0\quad \text{in}\,\, D,$$
unless $w \equiv 0$  in $D$.
\end{proposition}
\begin{proof}
The proof follows repeating verbatim the one in \cite{DMPS} that goes back to \cite[Proposition 2.17]{PhS}.
\end{proof}

\begin{remark}
Note that \eqref{sjdfbhswkvcgvcdgcvgdvc} has to be understood in the weak formulation. Since here we consider a domain $D$ that does not touch the pole, then there is no need to consider bounded test functions, as in Definition \ref{def:sol}.\end{remark}

\section{Symmetry of Solutions}

The main result of this section, that will be a consequence of more general  monotonicity results, see Proposition \ref {th:simmmmmmm} and Proposition \ref{th:simmmmmmmM} below,  is stated in the following
\begin{theorem}\label{symmetry}
Let $u\in \mathcal{C}^0(\overline{\Omega}\setminus\{0\})$ be a weak solution to \eqref{Eq:P} and let $\Omega$ be  convex with respect to the $\nu$-direction $(\nu \in S^{N-1})$ and symmetric w.r.t. $T_0^\nu$, where
$$ T_0^\nu=\{x\in \mathbb{R}^N : x\cdot \nu=0\}$$
and let us suppose that are fulfilled the following natural symmetry and monotonicity properties on $f(x,t)$:
\begin{itemize}
\item[$(*)$]  $ f(x,t)\leq f(x_\lambda^\nu,t)\quad \text{if } \lambda <0,\, x\in\Omega_\lambda^\nu,\, t\in [0,\infty);$

\

and

\

\item[$(**)$] $ f(x,t)= f(x_0^\nu,t)\quad \text{if } x\in\Omega_0^\nu,\, t\in [0,\infty).$
\end{itemize}

\noindent Let us also assume   assume that either $(a)$ or $(b)$ are fulfilled, where:

\begin{itemize}
\item [$(a)$]  $(H_1)$  holds,  $f(x,t)$  is  nondecreasing with respect to the variable $t$ for any $x\in \Omega$ and $g(t)$  is  nondecreasing with respect to the variable $t$;

\item[$(b)$]  $(H_1)$ and  $(H_2)$  holds.
\end{itemize}

\

\noindent Then $u$ is symmetric w.r.t.~$T_0^\nu$ and increasing w.r.t. the $\nu$-direction in~$\Omega_0^\nu$.
Moreover, if  $\Omega$ is a ball, then $u$ is radial and radially decreasing.
\end{theorem}
\begin{remark}
If e.g. we consider the case $\nu=(1,0,\cdots,0)$, hypotheses $(*)$ and $(**)$ in the statement of Theorem \ref{symmetry}  mean that $f(x,t)$ nondecreasing in the $x_1$-direction for $x_1<0$ and $f(x,t)$ even in $x_1$-direction respectively.
\end{remark}

Theorem \ref{symmetry} will be obtained as a consequence of  more general results. We start with the following

\begin{proposition}\label{th:simmmmmmm}
Let $u\in \mathcal{C}^{0}(\overline{\Omega}\setminus\{0\})$  be a {weak} solution to \eqref{Eq:P}. Set
\[
\lambda_1^0(\nu)\,:=\, \min \{0\,,\,\lambda_1(\nu)\},
\]
where $\lambda_1(\nu)$ is defined in \eqref{eq:sn7}. Assume that  $(H_1)$ is fulfilled and  assume that $f(x,t)$ is nondecreasing  with respect to the variable $t$, for all $x\in \Omega$, as well as $g(t)$ is nondecreasing with respect to the variable~$t$.
Assume also that
\begin{equation}\label{acabando}
f(x,t)\leq f(x_\lambda^\nu,t)\quad \text{if } \lambda \leq \lambda_1^0(\nu),\, x\in\Omega_\lambda^\nu,\,\,\, t\in [0,\infty).
\end{equation}
Then, for any $a(\nu)\leq\lambda\leq \lambda_1^0(\nu)$, we have
\begin{equation}\nonumber
u(x)\leq u^\nu_{\lambda}(x),\quad x \in \Omega^\nu _{\lambda}.
\end{equation}
Furthermore $u$ is   strictly monotone increasing in the $\nu$-direction in $\Omega^\nu _{\lambda}$.
\end{proposition}
\begin{proof}
In the proof we will fix  the direction $\nu=e_1:=(1,0,\ldots, 0)$. The proof can be carried out for general directions
with trivial modifications.
 In this case we have
\begin{equation}\label{eq:incontro0}
T_\lambda^\nu=T_\lambda=\{x\in \mathbb{R}^N:x_1=\lambda\}.
\end{equation}
Moreover
\begin{equation}\label{eq:incontro1}
\Sigma_\lambda^\nu=\Sigma_\lambda=\{x\in \mathbb{R}^N:x_1 <\lambda\},\quad \Omega_\lambda=\Omega\cap\Sigma_\lambda,
\end{equation}
\begin{equation}\label{eq:incontro11}
x_\lambda=R_\lambda(x)=(2\lambda-x_1,x_2,\ldots,x_n),
\end{equation}
\begin{equation}\label{bisseteeq:sn33}
u_{\lambda }(x)=u(x_\lambda) \,,
\end{equation}
and
\begin{equation}\label{bisseteeq:sn4}
a=\inf _{x\in\Omega}x_1.
\end{equation}
When $\lambda >a$, since $\Omega_\lambda$ is nonempty, we set
\begin{equation}\label{eq:colnumero}
(\Omega_\lambda)':=R_\lambda(\Omega_\lambda).
\end{equation}
We also denote in this case $\lambda_1=\sup\{\lambda : (\Omega_\lambda)'\subset \Omega\}$ and
$$\lambda_1^0\,:=\, \min \{0\,,\,\lambda_1\}\,.$$

For $a<\lambda<\lambda_1^0$, since $u$ is a solution to \eqref{Eq:P}, it is easy to verify that $u_\lambda\in H^s_0(R_\lambda(\Omega))$ satisfies
\begin{equation}\nonumber
\frac {c_{N,s}}{2}\int_{\R^N}\int_{\R^N}\frac{(u_\lambda(x)-u_\lambda(y))(\varphi(x)-\varphi(y))}{|x-y|^{N+2s}}\,dx\,dy=
\int_{R_\lambda(\Omega)}\frac{g(u_\lambda)}{|x_\lambda|^{2s}}\varphi \,dx+\int_{R_\lambda(\Omega)}f(x_\lambda,u_\lambda)\varphi\,dx,
\end{equation}
where $c_{N,s}$ has been defined in \eqref{constante} and  $\varphi \in H^s_0(R_\lambda(\Omega))\cap L^\infty (R_\lambda(\Omega))$. By \eqref{acabando} and the fact that $|x_\lambda|\leq |x|$ for $\lambda<0$, we deduce that
\begin{equation}\label{eqreflecata}
\frac {c_{N,s}}{2}\int_{\R^N}\int_{\R^N}\frac{(u_\lambda(x)-u_\lambda(y))(\varphi(x)-\varphi(y))}{|x-y|^{N+2s}}\,dx\,dy\geqslant
\int_{R_\lambda(\Omega)}\frac{g(u_\lambda)}{|x|^{2s}}\varphi \,dx+\int_{R_\lambda(\Omega)}f(x,u_\lambda)\varphi\,dx,
\end{equation}
for any nonnegative test function  $\varphi \in H^s_0(R_\lambda(\Omega))\cap  L^\infty (R_\lambda(\Omega))$.\\

\noindent We are now in position to exploit the weak comparison principle in small domains. In fact, for $\lambda-a$ small,  we can apply Theorem \ref{wcpsmall} with $\widetilde D=\Omega_\lambda$ and $v=u_\lambda$.
To control the behavior of the constants in Theorem \ref{wcpsmall}, we  fix $a<\hat \lambda<0$ so that
\[
\|u\|_{L^\infty(\Omega_\lambda)}\leqslant \|u\|_{L^\infty(\Omega_{\hat \lambda})}\qquad\text{for}\quad \lambda<\hat\lambda\,.
\]
This allows us to take $\lambda-a$ small enough to guarantee that $|\Omega_\lambda|\leqslant \delta$ where $\delta$ was given in \eqref{controlling}. Therefore by Theorem \ref{wcpsmall} we get that
$$
u\leqslant u_\lambda\qquad \text {in}\quad \Omega_\lambda\,\mbox{ for $\lambda<0$ such that $0<\lambda-a$ is small enough.}
$$
\\

We start now  the moving plane procedure setting
\begin{equation}\label{eq:LAMMMMMMMMBDA}
\Lambda\,:=\,\{a<\lambda<\lambda_1^0\,\,|\,\,u\leqslant  u_\mu\,\,\text{in}\,\,\Omega_\mu\quad\forall\,a<\mu\leqslant\lambda\}.
\end{equation}
In fact, as we already proved, we have that $\Lambda\ne\emptyset$ so we can set
\[
\bar\lambda\,:=\,\sup\,\Lambda\,.
\]
The proof of the theorem will be done if we show that $\bar\lambda=\lambda_1^0$. To prove this we argue by contradiction and we assume that $\bar\lambda<\lambda_1^0$. By continuity we deduce that
\begin{equation}\label{ssant2}
u\leqslant u_{\bar\lambda}\quad \text{in}\,\,\Omega_{\bar\lambda}\setminus\{0_{\bar\lambda}\}\,.
\end{equation}
Let us show that, in fact,
\begin{equation}\label{objetivoss}
u< u_{\bar\lambda}\quad \text{in}\,\,\Omega_{\bar\lambda}\setminus \{0_{\bar\lambda}\}\, \mbox{ for  $\bar\lambda<\lambda_1^0$}.
\end{equation}
To prove this note  that, since by assumption  $f(x,t)$ and $g(t)$ are nondecreasing with respect to the variable $t$, then it follows that
\begin{equation}\nonumber
\int_{\Omega_{\bar\lambda}}\frac{g(u)}{|x|^{2s}}\varphi \,dx+\int_{\Omega_{\bar\lambda}}f(x,u)\varphi\,dx\leqslant
\int_{\Omega_{\bar\lambda}}\frac{g(u_\lambda)}{|x|^{2s}}\varphi \,dx+\int_{\Omega_{\bar\lambda}}f(x,u_\lambda)\varphi\,dx,
\end{equation}
for  $0\leq\varphi \in H^s_0(\Omega_{\bar\lambda})\cap  L^\infty (\Omega_{\bar\lambda})$. From this and recalling \eqref{eqreflecata},
setting
\[
w_{\bar\lambda}\,:=\, u_{\bar\lambda}-u\,,
\]
it easy follows that
\begin{equation}\label{eq:intrinsecooo}
(-\Delta)^s\,w_\lambda \geq 0\quad \text{in} \,\, \Omega_{\bar\lambda}\,.
\end{equation}
For any $B_r(x)\subset\subset \Omega_{\bar\lambda}$ such that $0_{\bar\lambda}\notin B_r(x)$, since we have that $u$ and $u_{\bar\lambda}$ are continuous in $B_r(x)\subset\subset \Omega_{\bar\lambda}$, we can apply
  the strong comparison principle, given in Proposition \ref{proconfronto}, to deduce that $ u<u_{\bar\lambda}$ in $B_r(x)$ unless $u\equiv  u_{\bar\lambda}$ in $B_r(x)$. \\
 \\If now $ u<u_{\bar\lambda}$ in $\Omega_{\bar\lambda}\setminus \{0_{\bar\lambda}\}$ our claim holds true. If this is not the case, by \eqref{ssant2}, then there exists at least a point
  $\bar x\in \Omega_{\bar\lambda}\setminus \{0_{\bar\lambda}\} $ such that $ u(\bar x)=u_{\bar\lambda}(\bar x)$ and
  we can  consider $\sigma>0$ such that $\bar x\notin B_\sigma(0_{\bar\lambda})$. Since $u$ and $u_{\bar\lambda}$
  are continuous in the closure of $\Omega_{\bar\lambda}\setminus B_\sigma(0_{\bar\lambda})$, exploiting the strong comparison principle as here above, it follows
 that the set $\{u=u_{\bar\lambda}\}$ is not empty, open and closed in $\Omega_{\bar\lambda}\setminus B_\sigma(0_{\bar\lambda})$. This imply that $ u=u_{\bar\lambda}$ in the closure of $\Omega_{\bar\lambda}\setminus B_\sigma(0_{\bar\lambda})$.
   Since this is not possible  by the Dirichlet condition, then we have
\[
u<u_{\bar\lambda}\quad \text{in}\,\,\Omega_{\bar\lambda}\setminus \{0_{\bar\lambda}\}\,,
\]
that is, \eqref{objetivoss} follows.
\\

Let us now fix $\bar\e>0$ such that $\bar\lambda+\bar\e<\lambda_1^0$, so that
\[
\|u\|_{L^\infty(\Omega_{\bar\lambda+\e})}\leqslant \|u\|_{L^\infty(\Omega_{\bar\lambda+\bar\e})}\qquad\text{for any}\quad 0\leqslant \e \leqslant \bar\e\,.
\]
From this we will conclude that there exists $\bar\delta>0$, not depending on $\e$, such that the Theorem \ref{wcpsmall}  can be applied with some $\widetilde D=\widetilde D_\e$, $\widetilde D_\e\subset \Omega_{\bar\lambda+\e}$ such that
\begin{equation}\label{xx4}
|\widetilde D_\e|\leqslant \bar\delta\,,
 \end{equation}
$v=u_{\bar\lambda+\e}$, $\lambda=\bar\lambda+\e$ and  $0\leqslant\e\leqslant\bar\e $. In fact, in order to apply Theorem \ref{wcpsmall} we first
 consider $\hat\e>0$ such that $\hat\e<\bar\e$ and
\begin{equation}\label{xx1}
{|\Omega_{\bar\lambda+\hat\e}\setminus\Omega_{\bar\lambda}|}\leqslant \frac{\bar\delta}{4}\,,
\end{equation}
{for some $\bar\delta>0$ small enough.}
Furthermore we fix $\tau>0$ such that
\begin{equation}\label{xx2}
B_{4\tau}(0_{\bar\lambda})\subset{\Omega_{\bar\lambda}},\qquad\text{and}\qquad \mathcal |B_{4\tau}(0_{\bar\lambda})|\leqslant \frac{\bar\delta}{4},
\end{equation}
and
we consider a compact set $K$ so that
\begin{equation}\label{xx3}
K\subset \left(\Omega_{\bar\lambda}\setminus B_{4\tau}(0_{\bar\lambda})\right)\qquad\text{and}\qquad \left|(\Omega_{\bar\lambda}\setminus B_{4\tau}(0_{\bar\lambda})\setminus K\right|\leqslant \frac{\bar\delta}{4}.
\end{equation}
Since
$u$ is continuous  in the interior of $\Omega\setminus \{0\}$, it follows that there exists $\rho=\rho(K)>0$ such that
\[
w_{\bar\lambda}\geqslant \rho,\quad\text{in}\quad K\,.
\]
We assume now, without loss of generality, that $\hat\e<\tau$, thus obtaining
\[
0_{\bar\lambda+\e}\subset B_{\tau}(0_{\bar\lambda}),\quad\text{for all }\quad 0\leqslant \e<\hat\e\,.
\]
This allows us to exploit the fact that $u$ is uniformly continuous in $\overline{\Omega\setminus \{B_{2\tau}(0)\}}$ to deduce that,
eventually reducing $\hat\e$, we have
\begin{equation}\label{acabando2}
w_{\bar\lambda+\e}\geqslant \frac{\rho}{2}\quad\text{in $K$,  $0\leqslant \e<\hat\e.$}
\end{equation}
Setting
\[
\widetilde D_\e\,:=\, \Omega_{\bar\lambda+\e}\setminus K,\quad 0\leqslant \e<\hat\e\,,
\]
it follows  by \eqref{xx1}, \eqref{xx2} and \eqref{xx3} that $|\widetilde D_\e|<\bar\delta$ ( i.e. \eqref{xx4} follows).  Applying Theorem~\ref{wcpsmall}  and using \eqref{acabando2} we obtain that
$$u\leqslant u_{\bar\lambda+\e}\text{ in $\Omega_{\bar\lambda+\e}$, $0<\varepsilon\leqslant\hat\varepsilon<\bar\e$}.$$
This is a contradiction with the definition of $\bar\lambda$ showing that
the case $\bar\lambda<\lambda_1^0$ cannot occur.\\

Let us finally prove that the solution  is strictly monotone increasing in $\Omega_{\lambda_1^0}$.\\
For $(t_1,x')$ and $(t_2,x')$ belonging to $\Omega_{\lambda_1^0}$, with $t_1<t_2$, we have that
$u\leqslant u_{\frac{t_1+t_2}{2}}$ in $\Omega_{\frac{t_1+t_2}{2}}$.  Actually, by the strong comparison principle
and arguing as above, it follows that $u< u_{\frac{t_1+t_2}{2}}$ in $\Omega_{\frac{t_1+t_2}{2}}$.
This implies that
\[
u(t_1,x')<u(t_2,x')\,,
\]
and we have done.

\end{proof}

Now we prove a similar result, but under a different set of assumptions. Namely we have the following

\begin{proposition}\label{th:simmmmmmmM}
Let $u\in \mathcal{C}^{0}(\overline{\Omega}\setminus\{0\})$  be a {weak} solution to \eqref{Eq:P} and assume that $(H_1)$ and  $(H_2)$  hold. Set
\[
\lambda_1^0(\nu)\,:=\, \min \{0\,,\,\lambda_1(\nu)\},
\]
where $\lambda_1(\nu)$ is defined in \eqref{eq:sn7}.  Assume also that
\begin{equation}\label{sunday}
 f(x,t)\leq f(x_\lambda^\nu,t)\quad \text{if } \lambda \leq \lambda_1^0(\nu),\, x\in\Omega_\lambda^\nu,\, t\in \,\,[0,\infty).
 \end{equation}
Then, for any $a(\nu)\leq\lambda\leq \lambda_1^0(\nu)$, we have
$$
u(x)\leq u^\nu_{\lambda}(x),\quad  x \in \Omega^\nu _{\lambda}.
$$
Furthermore $u$ is strictly monotone   increasing in the $\nu$-direction in $\Omega^\nu _{\lambda}$.
\end{proposition}
\begin{proof}
As in the proof of Proposition \ref{th:simmmmmmm}, we will fix  the direction $\nu=(1,0,\ldots,0)$ without loss of generality.  We refer also to the same notations in such proof, in particular  see equations \eqref{eq:incontro0}, \eqref{eq:incontro1}, \eqref{eq:incontro11}, \eqref{bisseteeq:sn33} and \eqref{bisseteeq:sn4}.
When $\lambda >a$, since $\Omega_\lambda$ is nonempty, as in \eqref{eq:colnumero}, we define
$$(\Omega_\lambda)':=R_\lambda(\Omega_\lambda),\quad \lambda_1=\sup\{\lambda : (\Omega_\lambda)'\subset \Omega\},$$
and $\lambda_1^0\,:=\, \min \{0\,,\,\lambda_1\}\,$.

\noindent The first part of the proof relies  on the proof of Proposition \ref{th:simmmmmmm} that is, for $\lambda-a$ small, we will apply Theorem \ref{wcpsmall}. We fix $a<\hat \lambda<0$ so that
$$\|u\|_{L^\infty(\Omega_\lambda)}\leqslant \|u\|_{L^\infty(\Omega_{\hat \lambda})}\qquad\text{for}\quad \lambda<\hat\lambda\,.$$
This allows us to take $\lambda<\hat\lambda$ with $\lambda-a$ small enough such that $|\Omega_\lambda|\leqslant \delta$ where $\delta$ was given in \eqref{controlling}. Therefore, applying Theorem \ref{wcpsmall} with $\widetilde D=\Omega_\lambda$ and $v=u_\lambda$, we get that  $u\leqslant u_\lambda$ {in} $\Omega_\lambda$.

We start now the moving plane procedure setting
\[
\bar\lambda\,:=\,\sup\,\Lambda,
\]
where $\Lambda\neq\emptyset$, was given in \eqref{eq:LAMMMMMMMMBDA}. The proof of the theorem will be done if we show that $\bar\lambda=\lambda_1^0$. As in the proof of Proposition \ref{th:simmmmmmm} we argue by contradiction so we suppose that
$\bar\lambda<\lambda_1^0$. By continuity we deduce that $
u\leqslant u_{\bar\lambda}$ in $\Omega_{\bar\lambda}\setminus\{0_{\bar\lambda}\}$.
Let us show that
\begin{equation}\label{eq:aiaiaiaiiaiaiaiaiia}
u< u_{\bar\lambda}\quad \text{in}\,\,\Omega_{\bar\lambda}\setminus \{0_{\bar\lambda}\}.
\end{equation}
We point out that the case $u\equiv u_{\bar\lambda}$  in $\Omega_{\bar\lambda}\setminus \{0_{\bar\lambda}\}$ is not possible by the Dirichlet condition. \\
\noindent Therefore, to prove \eqref{eq:aiaiaiaiiaiaiaiaiia}, we assume by contradiction that there exists a point $\bar x$ in $\Omega_{\bar\lambda}\setminus \{0_{\bar\lambda}\}$ where
\begin{equation}\label{eq:contrddd}
u(\bar x)=u_{\bar \lambda}(\bar x).
\end{equation}
We fix now  $r>0$ such that $0 \notin \overline{B_r}(\bar x )$ and $0_\lambda \notin \overline{B_r}(\bar x )$. Then using  Proposition \ref{pr:regularity} we have that there exists $0<\alpha<2s$ such that
$$
\|u\|_{\mathcal{C}^{\alpha+2s}\left(\overline{B_{\frac{r}{16}}}(\bar x)\right)}\leq C\quad  \text {and } \quad \|u_\lambda\|_{\mathcal{C}^{\alpha+2s}\left(\overline{B_{\frac{r}{16}}}(\bar x)\right)}\leq C,
$$
hold, for some positive constant
$C=C\left(\lambda, f, g,\|u\|_{L^\infty(\Omega_{\hat \lambda})}\right)$.
\\As consequence, see Remark \ref{rm:point}, we can write the pointwise formulation of the problem \eqref{Eq:P} for both $u$ and $u_\lambda$ in the point $x=\bar x$. Therefore
\begin{equation}\label{eq:1w_l}
(-\Delta)^su_\lambda(\bar x)-(-\Delta)^su(\bar x)=\frac{g(u_\lambda(\bar x))}{|\bar x_\lambda|^{2s}}-\frac{g(u(\bar x))}{|\bar x|^{2s}}+f(\bar x_\lambda,u_\lambda(\bar x))-
f(\bar x, u(\bar x)).
\end{equation}
{Then, from the previous equation, if $g\equiv 0$, by \eqref{sunday} and \eqref{eq:contrddd}, we obtain that
\begin{equation}\label{eq:1w_l1}
(-\Delta)^su_\lambda(\bar x)-(-\Delta)^su(\bar x)\geq 0.
\end{equation}
It is worth noticing that, if $g\not \equiv0$, using also \eqref{sunday}, \eqref{eq:contrddd} and the fact that $|\bar x_\lambda|<|\bar x|$ for $\lambda <0$, from \eqref{eq:1w_l} it follows that
\begin{equation}\label{eq:1w_l2}
(-\Delta)^su_\lambda(\bar x)-(-\Delta)^su(\bar x)> 0.
\end{equation}}

\noindent On the other hand, by \eqref{org1}, \eqref{eq:contrddd} and the fact that the function $u_\lambda(y)-u(y)$ is odd with respect to the hyperplane $\partial \Sigma_\lambda=T_\lambda$, it follows that
\begin{eqnarray}\nonumber
(-\Delta)^su_\lambda(\bar x)-(-\Delta)^su(\bar x)&=&c_{N,s}\,{\rm P.V.}\int_{\mathbb{R}^N}\frac{u_\lambda(\bar x)-u_\lambda(y)}{|\bar x-y|^{N+2s}}dy-c_{N,s}\,{\rm P.V.}\int_{\mathbb{R}^N}\frac{u(\bar x)-u(y)}{|\bar x-y|^{N+2s}}dy\\\nonumber
&=&-c_{N,s}\,{\rm P.V.}\int_{\mathbb{R}^N}\frac{u_\lambda(y)-u(y)}{|\bar x-y|^{N+2s}}dy\\\nonumber&=&-c_{N,s}\,{\rm P.V.}\int_{\Sigma_\lambda}\frac{u_\lambda(y)-u(y)}{|\bar x-y|^{N+2s}}dy-c_{N,s}\,{\rm P.V.}\int_{\mathbb{R}^N\setminus\Sigma_\lambda}\frac{u_\lambda(y)-u(y)}{|\bar x-y|^{N+2s}}dy
\\\nonumber&=&-c_{N,s}\,{\rm P.V.}\int_{\Sigma_\lambda}\frac{u_\lambda(y)-u(y)}{|\bar x-y|^{N+2s}}dy-c_{N,s}\,{\rm P.V.}\int_{\Sigma_\lambda}\frac{u_\lambda(y_\lambda)-u(y_\lambda)}{|\bar x-y_\lambda|^{N+2s}}dy\\\nonumber
&=&-c_{N,s}\,{\rm P.V.}\int_{\Sigma_\lambda}(u_\lambda(y)-u(y))\left(\frac{1}{|\bar x-y|^{N+2s}}-\frac{1}{|\bar x-y_\lambda|^{N+2s}}\right)dy.\\\label{Auguri}
\end{eqnarray}
Since $|\bar x-y|\leq|\bar x -y_\lambda|$ for $x,y\in \Sigma_\lambda$ and $u\not \equiv u_\lambda$, $u\leq u_\lambda$, from \eqref{Auguri}, {by continuity}, we have
\begin{equation}\label{eq:giuseppestronzo}
(-\Delta)^su_\lambda(\bar x)-(-\Delta)^su(\bar x)\,<\,0,\end{equation}
a contradiction with \eqref{eq:1w_l1}{(and \eqref{eq:1w_l2}.)}
\\

 Following verbatim the second part of the  proof of Proposition \ref{th:simmmmmmm} we conclude that,
\begin{equation}\nonumber
u(x)\leq u_{\lambda}(x),\quad  x \in \Omega _{\lambda},\quad a\leq\lambda\leq \lambda_1^0\,.
\end{equation}

Let us finally prove that the solution  is strictly monotone increasing in $\Omega_{\lambda_1^0}$. For $(t_1,x')$ and $(t_2,x')$ belonging to $\Omega_{\lambda_1^0}$, with $t_1<t_2$, we have that
$u\leqslant u_{\frac{t_1+t_2}{2}}$ in $\Omega_{\frac{t_1+t_2}{2}}$.  Actually, arguing as above (see equations \eqref{eq:contrddd}-\eqref{eq:giuseppestronzo}), it follows that $u< u_{\frac{t_1+t_2}{2}}$ in $\Omega_{\frac{t_1+t_2}{2}}$. Thus we deduce that
$$u(t_1,x')<u(t_2,x')\,.$$
\end{proof}

\begin{proof}[Proof of Theorem \ref{symmetry}]
 Since by hypothesis $\Omega$ is convex w.r.t. the $\nu$-direction  and symmetric w.r.t. to
$$ T_0^\nu=\{x\in \mathbb{R}^N:x\cdot\nu=0\},$$
then  $\lambda_1(\nu)=0=\lambda_1^0(\nu)$. Therefore by Proposition \ref{th:simmmmmmm} in the case of assumptions  $(a)$ or  by Proposition \ref{th:simmmmmmmM} in the case of assumptions  $(b)$, in the statement, one has
$$u(x)\leq u_\lambda^\nu(x) \,\,  \text{ for } x\in \Omega_0^\nu.$$
In the same way, performing the moving plane method in the opposite  direction $-\nu$ we obtain
$$u(x)\geq u_\lambda^\nu(x) \,\,  \text{ for } x \in \Omega_0^\nu,$$that is, $u$ is symmetric  and non decreasing w.r.t. the  $\nu$-direction in $\Omega_0^\nu$, since monotonicity follows directly from Proposition \ref{th:simmmmmmm} and Proposition \ref{th:simmmmmmmM}.

\

\noindent Finally, if $\Omega$ is a ball, repeating this argument along any direction, it follows that $u$ is radially symmetric, i.e. $u=u(r)$ and strictly decreasing w.r.t $r$.
\end{proof}

\begin{proof}[Proof of Theorem \ref{symmetryintro}]
The proof of Theorem \ref{symmetryintro} follows by  Theorem \ref{symmetry} considering there the case $\nu=(1,0,\cdots,0)$.
\end{proof}

\noindent \bf{Acknowledgements}:
\rm{We gratefully thank Prof. Ireneo Peral for the conversations about the topic of this work and for his suggestions.   }

\end{document}